\documentclass[12pt,a4paper]{amsart}
\usepackage{mathrsfs}
\usepackage{todonotes}
\usepackage[utf8]{inputenc}
\usepackage{amssymb}
\usepackage{amsmath}
\usepackage{amsfonts}
\usepackage{amsthm}
\usepackage{tikz}
\usepackage{hyperref}
\usepackage{yhmath}
\usepackage{mathtools}
\usepackage{dsfont}
\usepackage{accents}
\usepackage{upgreek}
\usepackage{enumitem}
\newtheorem{theorem}{Theorem}[section]
\newtheorem{claim}[theorem]{Claim}

\newtheorem{lemma}[theorem]{Lemma}

\theoremstyle{definition}
\newtheorem{definition}[theorem]{Definition}

\theoremstyle{remark}

\newcommand{\on}{{\rm Ord}}
\newcount\skewfactor
\def\mathunderaccent#1#2 {\let\theaccent#1\skewfactor#2
\mathpalette\putaccentunder}
\def\putaccentunder#1#2{\oalign{$#1#2$\crcr\hidewidth
\vbox to.2ex{\hbox{$#1\skew\skewfactor\theaccent{}$}\vss}\hidewidth}}

\title[Adding  Abraham clubs and  $\alpha$-properness]{Adding  Abraham clubs and  $\alpha$-properness}
\author{Mohammad Golshani}
\address{School of Mathematics\\
Institute for Research in Fundamental Sciences (IPM)\\
P.O.Box:
19395-5746\\
Tehran-Iran.}
\email{golshani.m@gmail.com}
\urladdr{https://math.ipm.ac.ir/~golshani/}
\author{Rouholah Hoseini Naveh}
\address{Department of Pure Mathematics\\
Faculty of Mathematics \& Computer\\
Shahid Bahonar University of Kerman\\
Kerman, Iran}
\email{r.hoseini.nave@Gmail.com}
\urladdr{https://rhoseininaveh.github.io}
\thanks{The first author's research has been supported by a grant from IPM (No.1403030417)}
\thanks{The second author's research is partially supported by FWF (P33420), hosted by Jakob Kellner}
\begin{document}
\begin{abstract}
For every indecomposable ordinal $\alpha < \omega_1$, we introduce a variant of Abraham forcing for adding a club in $\omega_1$,
which is $<\alpha$-proper but not $\alpha$-proper.
\end{abstract}
\makeatother
\maketitle
\setcounter{section}{-1}
\section{Introduction}
To preserve $\aleph_1$, the properness property for a forcing notion was introduced by Shelah during his initial investigation of countable support iterations, see \cite{shelah80}. A poset $\mathbb{P}$ is called proper if forcing with $\mathbb{P}$  preserves stationary subsets of $[\lambda]^{\aleph_0}$, for all uncountable regular cardinals $\lambda$. If $\mathbb{P}$ is proper, then every countable set of ordinals in the extension is covered by a countable set of ordinals from the ground model, thus in particular forcing with $\mathbb{P}$ does not collapse $\aleph_1$. Shelah also introduced a characterization of properness, using $(N, \mathbb{P})$-generic conditions. The forcing notion $\mathbb{P}$ is proper if and only if for every large enough regular cardinals $\lambda$, and for club many countable elementary substructure $N$ of $H(\lambda)$, for every condition $p \in N$, there is $q \leq p$ which is an $(N, \mathbb{P})$-generic condition, i.e., for every dense open subset $D \subseteq \mathbb{P}$ with $D \in N$, $D \cap N$ is predense below $q$.

Afterwards, Mitchell introduced strong properness using strongly $(N, \mathbb{P})$-generic conditions, which differs from $(N, \mathbb{P})$-generic conditions only in the point that every dense open $D \subseteq \mathbb{P} \cap N$ must be predense below that condition \cite{mitchell}. Obviously a notion of forcing is proper, if it is strongly proper.

Shelah also introduced, for each $\alpha<\omega_1$, a technical strengthening of properness known as ``$\alpha$-properness'',
and showed for every indecomposable ordinal $\alpha,$ there is a forcing notion which is $<\alpha$-proper
but  not $\alpha$-proper.

In this paper,  for every indecomposable ordinal $\alpha,$  we introduce a variant of Abraham forcing from \cite{AS} for adding a club in $\omega_1$, which is $<\alpha$-proper, but not $\alpha$-proper, so giving a new example for separation $\alpha$-properness.

In section \ref{2} and inspired by Neeman's approach \cite{neeman}, we introduce a variant of Abraham forcing for adding a club in $\omega_1$ with finite conditions and show that it is strongly proper, but  not  $\omega$-proper. Then in section \ref{3}, for every indecomposable ordinal $\alpha < \omega_1,$ we introduce  a generalization  $\mathbb{P}[\alpha]$ of our previous forcing  which adds a club in $\omega_1$, and such that  $\mathbb{P}[\alpha]$ is $< \alpha$-proper but not $\alpha$-proper. Indeed the forcing has the stronger property that if $\beta < \alpha$ and if $\mathcal{N}=\langle N_\xi: \xi \leq \beta \rangle$ is a $\beta$-tower with $\mathbb{P}[\alpha] \in N_0,$ then  every condition $p \in N_0$ has an extension $q$ such that $q$ is strongly $(N_{\xi}, \mathbb{P}[\alpha])$-generic, whenever $\xi \leq \beta$ is not a limit ordinal, and is $(N_{\delta}, \mathbb{P}[\alpha])$-generic, whenever $\delta \leq \beta$ is a limit ordinal.
\section{preliminaries}
For a regular cardinal $\lambda$, let $H(\lambda)$ denote the collection of all sets $x$, whose transitive closure has size less than $\lambda$. We
work with  the structure $\langle H(\lambda), \in, \lhd^* \rangle$, where $\lhd^*$  is a fix well ordering of $H(\lambda)$. We use $N \prec H(\lambda)$ if $\langle N, \in, \lhd^* \rangle$ is an elementary substructure of $\langle H(\lambda), \in, \lhd^* \rangle$ and show the ordinal $N \cap \omega_1$ by $\delta_N$. Let also $\mathcal{S}$ represents the collection of all countable elementary substructure of $H(\omega_1)$.

Let $\mathbb{P}$ be a notion of forcing and let $\lambda$ be a regular cardinal. We say $\lambda$ is large enough (with respect to $\mathbb{P}$), if $\lambda > (2^{|\text{tr}(\mathbb{P})|})^+$, where $\text{tr}(\mathbb{P})$ is the transitive closure of $\mathbb{P}$. Note that if $\lambda$ is large enough, then all interesting statements about $\mathbb{P}$ are absolute between $H(\lambda)$ and the ground model $V$.
\begin{definition}\cite{mitchell}
Let $P$ be a notion of forcing and let $\lambda$ be a large enough regular cardinal. Let $N \prec H(\lambda)$ with $\mathbb{P} \in N$. A condition $q \in \mathbb{P}$ is said to be strongly $(N, \mathbb{P})$-generic if every dense open subset $D \subseteq \mathbb{P} \cap N$ is predense below $q$, i.e. for every $r \in \mathbb{P}$, if $r \leq q$ then there are $s \in \mathbb{P}$ and $t \in D \cap \mathbb{P}$ such that $s$ extends both $r$ and $t$.
\end{definition}
\begin{definition}\cite{todorcevic3}
A notion of forcing $\mathbb{P}$ is called strongly proper if for every large enough regular cardinal $\lambda$ and club many countable $N \prec H(\lambda)$ with $\mathbb{P} \in N$, every $p \in \mathbb{P} \cap N$ has a strongly $(N, \mathbb{P})$-generic extension.
\end{definition}
The next lemma is evident.
\begin{lemma}
If $\mathbb{P}$ is strongly proper, then $\mathbb{P}$ is proper, in particular forcing with $\mathbb{P}$ does not collapse $\aleph_1$.
\end{lemma}
\begin{definition}\cite{shelah80}
Let $\alpha< \omega_1$. The sequence $\mathcal{N}=\langle N_{\xi}\colon \xi \leq \alpha \rangle$ is said to be an $\alpha$-tower if for some regular cardinal $\lambda$,
\begin{enumerate}
\item
$N_{\xi}$ is a countable elementary substructure of $H(\lambda)$ for all $\xi \leq \alpha$;
\item
$\alpha \in N_0$;
\item
$N_{\zeta} \in N_{\zeta +1}$ for all $\zeta < \alpha$;
\item
$N_{\delta}= \bigcup_{\xi < \delta}N_{\xi}$ for all limit ordinals $\delta \leq \alpha$;
\item
$\langle N_{\zeta}: \zeta \leq \xi \rangle \in N_{\xi +1}$ for every $\xi < \alpha$.
\end{enumerate}
\end{definition}
The notion of $\alpha$-properness is defined as follows.
\begin{definition}\cite{shelah80}
Assume $\alpha< \omega_1$, and $\mathbb{P}$ is a forcing notion. $\mathbb{P}$ is called $\alpha$-proper if for every $\alpha$-tower
$\mathcal{N}=\langle N_{\xi}\colon \xi \leq \alpha \rangle$ with $\mathbb{P} \in N_0$, every condition $p \in \mathbb{P} \cap N_0$ has an extension $q$ which is $(\mathcal{N}, \mathbb{P})$-generic, i.e., $q$ is an $(N_\xi, \mathbb{P})$-generic condition for each $\xi \leq \alpha$.
$\mathbb{P}$ is called $<\alpha$-proper, if it is $\beta$-proper for each $\beta < \alpha.$
\end{definition}
\section{Adding a club in $\omega_1$ by finite conditions}\label{2}
In this section we introduce a variant of Abraham forcing \cite{AS}, and show that it is strongly proper but not $\omega$-proper. We start by defining our forcing notion $\mathbb{P}$.
\begin{definition}
Let  $\mathbb{P}$ consist of pairs $p=\langle \mathcal{M}_p, f_p \rangle$, where
\begin{enumerate}
\item
$\mathcal{M}_p= \langle M^p_i: i < n_p  \rangle$ is a finite $\in$-increasing sequence of elements of $\mathcal{S}$, and
\item
the function $f_p: \mathcal{M}_p \longrightarrow H(\omega_1)$ is defined such that $f_p(M^p_i)$ is a finite subset of $M^p_{i+1}$ if $i<n_p-1$, and $f_p(M^p_{n_p-1})$ is a finite subset of $H(\omega_1)$.
\end{enumerate}
For $p, q \in \mathbb{P}$, we say $q \leq p$ if and only if $\mathcal{M}_p \subseteq \mathcal{M}_q$ and $f_p(M) \subseteq f_q(M)$ for every $M \in \mathcal{M}_p$.
\end{definition}
\begin{lemma}
If $G \subseteq \mathbb{P}$ is a generic filter, then $C=\{\delta_M \colon M \in \mathcal{M}_p \text{ for some } p \in G \}$ is a club in $\omega_1$.
\end{lemma}
\begin{proof}
First, we need to prove the following claim:
\begin{claim} \label{extension}
For every $p \in \mathbb{P}$ and $\gamma \in \omega_1$, there is  $p' \leq p$ such that $\gamma < \delta_{N}$ for some $N \in \mathcal{M}_{p'}$.
\end{claim}
\begin{proof}[Proof of the Claim] Since $p, \gamma \in H(\omega_1)$, we can find $N \prec H(\omega_1)$ such that $p, {\gamma} \in N$. Let $p'=\langle \mathcal{M}_{p'}, f_{p'}   \rangle$ be such that $\mathcal{M}_{p'}= \mathcal{M}_p \cup \{N\}$, $f_{p'}(M)=f_p(M)$ for every $M \in \mathcal{M}_p$ and $f_{p'}(N)=\emptyset$. It is easy to check that $p'$ is a condition which extends $p$ and has the desired property.
\end{proof}
Given any $\gamma \in \omega_1$, by Claim \ref{extension},  the  set
$$D_{\gamma}= \{q \in \mathbb{P} \colon \exists M \in \mathcal{M}_q(\gamma < \delta_M)\}$$
 is dense  in $\mathbb{P}$. This implies that  $C$ is unbounded in $\omega_1$.

 Now we claim that for every $p \in \mathbb{P}$ and $\gamma \in \omega_1$, if $p$ forces $\gamma$ to be a limit point of $\dot{C}$, then $p$ also forces it is an element of $\dot{C}$. By contradiction suppose $p$ does not hold in the statement. Hence there is no $M \in \mathcal{M}_p$ with $\delta_{M}= \gamma$.
 Using Claim  \ref{extension} and the assumption that $p$ forces $\gamma$ is a limit point of $\dot{C}$,    by extending $p$ if necessary,  we can assume that
 for some $i$ with $i+1 < n_p$, $\delta_{M^p_i} < \gamma < \delta_{M^p_{i+1}}.$
  Let $\xi < \delta_{M^p_{i+1}}$ be any ordinal greater than $\gamma$. Set $q= \langle \mathcal{M}_q, f_q \rangle$ where $\mathcal{M}_q= \mathcal{M}_p$ and $f_q(M)=f_p(M)$ for all $M \neq {M^p_{i}}$, and $f_q({M^p_{i}})=f_p({M^p_{i}}) \cup \{\xi\}$. Now we have $q \leq p$ is a condition and  every extension of $q$ forces that $\gamma$ is not a limit point of $\dot{C}$, indeed for any extension $r$ of $q$, $r$ forces $\dot{C}$ has empty intersection with the interval $(\delta_{M^p_{i}}, \xi)$, in particular $r$ forces $\dot{C} \cap \gamma \subseteq \delta_{M^p_{i}}+1.$
\end{proof}
\begin{lemma}
$\mathbb{P}$ is strongly proper.
\end{lemma}
\begin{proof}
Suppose that $\lambda$ is a large enough regular cardinal, $N \prec H(\lambda)$ is countable with $\mathbb{P} \in N$, and $p \in \mathbb{P} \cap N$.
 Set $N'=N \cap H(\omega_1)$ and $p'=\langle \mathcal{M}_{p'}, f_{p'}  \rangle$, where
 $\mathcal{M}_{p'} = \mathcal{M}_p \cup \{N'\}$ and $f_{p'}(M) = f_p(M)$ for all $M \in \mathcal{M}_p$ and $f_{p'}(N')=\emptyset$.

We demonstrate that $p'$ serves as a strongly $(N, \mathbb{P})$-generic extension of $p$. It is clear that $p'$ is a condition which extends $p$. Now we argue that, if $q \leq p'$, then $q{\restriction}_N= \langle \mathcal{M}_q \cap N, f_q {\restriction}_{\mathcal{M}_q \cap N} \rangle \in \mathbb{P} \cap N$.
\begin{claim}
$q{\restriction}_N \in \mathbb{P} \cap N$.
\end{claim}
\begin{proof}[Proof of the Claim] First we show that $\mathcal{M}_q \cap N$ is exactly the initial segment of $\mathcal{M}_q$ before $N'$. Since the sequence of elements of $\mathcal{M}_q$ before $N'$ is a finite subset of $N' \subseteq N$, so it is also a finite subset of $N$. If there is $Q \in \mathcal{M}_q$ such that $N' \in Q$, then $Q \notin N$, since otherwise $Q \in N \cap H(\omega_1)=N'$ which gives a contradiction.
The result follows immediately.
\end{proof}
Now let $D \subseteq \mathbb{P} \cap N$ dense open and let $r \in D$ be an extension of $q{\restriction}_N$. We have to find a condition $s$ such that $s \leq r, q$.
\begin{claim}
Let $\mathcal{M}_s= \mathcal{M}_r \cup \mathcal{M}_q$, $f_s(M)=f_r(M)$ if $M \in \mathcal{M}_r$ and $f_s(M)= f_q(M)$ if $M \in \mathcal{M}_q \setminus \mathcal{M}_r$. Then $s$ is a common extension of $q$ and $r$.
\end{claim}
\begin{proof}[Proof of the Claim]  $\mathcal{M}_s$ is a finite increasing sequence of elements of $\mathcal{S}$, since $\mathcal{M}_q \cap N \subseteq \mathcal{M}_r \subseteq N'$ and $N' \in Q$ for every $Q \in \mathcal{M}_q \setminus \mathcal{M}_r$.
It is also easily seen that $f_s$ has the property that $f_s(M^s_i) \in M^s_{i+1}$ for every $i+1<n_s$. Thus $s$ is a condition, and it clearly extends both of $q$ and $r$.
\end{proof}
Therefore, $s$ is a witness for $D$ to be predense below $p'$, and $\mathbb{P}$ is strongly proper.
\end{proof}
\begin{lemma}
$\mathbb{P}$ is not  $\omega$-proper.
\end{lemma}
\begin{proof}
Assume that $\mathcal{N}=\langle N_i \colon i \leq \omega \rangle$ is an $\omega$-tower of elements of $\mathcal{S}$ such that $\dot{C},\mathbb{P} \in N_0$, where $\dot{C}$ is the canonical name for the club $C$,  and let  $p \in \mathbb{P} \cap N_0$. We show that there is no extension  $q \leq p$ which is an $(\mathcal{N}, \mathbb{P})$-generic condition, meaning that $q$  is $(N_i, \mathbb{P})$-generic for every $i \leq \omega$. Suppose by the way of contradiction that there is such a condition $q$.  For each $i \leq \omega$, as  $q$ is  $(N_i, \mathbb{P})$-generic and
$\dot{C} \in N_i$, we have $q \Vdash$``$\dot{C} \cap \delta_{N_i}$ is unbounded in $\delta_{N_i}$''.  As $q \Vdash$``$\dot{C}$ is a club in $\omega_1$'',   hence
 $q \Vdash$``$ \delta_{N_i} \in \dot{C}$''.

 On the other hand, by claim \ref{extension} and since $\mathcal{M}_q$ is finite, there are $q' \leq q$, $n \in \omega$, and  $i+1 < n_{q'}$ such that $\delta_{N^{q'}_i} < \delta_{N_n} < \delta_{N_{\omega}} \leq \delta_{N^{q'}_{i+1}}$.
Since $f_{q'}(N^{q'}_i)$ is a finite subset of $N^{q'}_{i+1}$, there exist $m \in \omega$ and $\delta_{N_m} < \xi <\delta_{N^{q'}_{i+1}}$ such that $n < m$ and $\xi \notin f_{q'}(N^{q'}_i)$. Let $q''$ be such that $\mathcal{M}_{q''}= \mathcal{M}_{q'}$ and $f_{q''}(M)= f_{q'}(M)$ for all $M \neq N^{q'}_{i}$ and $f_{q''}(N^{q'}_{i})= f_{q'}(N^{q'}_{i}) \cup \{\xi\}$. As before, $q'' \in \mathbb{P}$ is  an extension of $q$ such that for all $r \leq q''(r \Vdash `` \delta_{N_m} \notin \dot{C}$''$)$ which leads us to a contradiction.
\end{proof}
\section{General case}\label{3}
In this section, we present a new proof for a theorem by Shelah, which states that for every indecomposable countable ordinal $\alpha$, there exists a forcing notion $\mathbb{P}[\alpha]$, such that $\mathbb{P}[\alpha]$ is $\beta$-proper for every $\beta < \alpha$, but is not $\alpha$-proper.
\begin{definition}
Let $\alpha < \omega_1$ be an indecomposable ordinal. The forcing notion $\mathbb{P}[\alpha]$ consists of conditions $p=\langle \mathcal{M}_p, f_p, \mathcal{W}_p \rangle$ such that:
\begin{itemize}
\item
$\mathcal{M}_p=\langle M^p_{\xi} \colon \xi \leq \gamma_p \rangle$ is an $\in$-increasing sequence of elements of $\mathcal{S}$ for some $\gamma_p < \alpha$ which is continuous at limits;
\item
the function $f_p: \mathcal{M}_p \longrightarrow H(\omega_1)$ is defined such that $f_p(M^p_{\xi})$ is a finite subset of $M^p_{\xi+1}$ for $\xi < \gamma$, and $f_p(M^p_{\gamma})$ is a finite subset of $H(\omega_1)$; and
\item
the witness $\mathcal{W}_p$ is a subset of $\mathcal{M}_p,$ such that for every $N \in \mathcal{W}_p$, $p{\restriction}_N= \langle \mathcal{M}_p \cap N, f_p{\restriction}_{\mathcal{M}_p \cap N}, \mathcal{W}_p \cap N \rangle \in N$.
\end{itemize}
For $p, q \in \mathbb{P}[\alpha]$, we say $q \leq p$ if and only if $\mathcal{M}_p \subseteq \mathcal{M}_q$, $f_p(M) \subseteq f_q(M)$ for every $M \in \mathcal{M}_p$ and $\mathcal{W}_p \subseteq \mathcal{W}_q$.
\end{definition}
\begin{lemma}
If $G \subseteq \mathbb{P}[\alpha]$ is a generic filter, then $C=\{\delta_M \colon M \in \mathcal{M}_p \text{ for some } p \in G\}$ is a club.
\end{lemma}
\begin{proof}
Let $\gamma \in \omega_1$. Similar to the proof of claim \ref{extension}, for every $p \in \mathbb{P}[\alpha]$ we can find $q \leq p$ such that for some  $N \in \mathcal{M}_q$, $\gamma < \delta_N$, so the set $D_{\gamma}= \{q \in \mathbb{P}[\alpha] \colon \exists \xi \leq \gamma_q(\gamma \leq \delta_{M_{\xi}})\}$ is dense open, which guaranties that $C$ is unbounded in $\omega_1$.

Now assume that $p \Vdash$``$\gamma \notin \dot{C}$''. We show that $p$ forces that $\gamma$ can not be a limit point of $C$. For simplicity write
$M_\eta=M^p_\eta$, for all $\eta \leq \gamma_p.$
First, by extending $p$ let us assume that  $\gamma < \delta_{M_{\gamma_p}}$.
Let
$$ \delta = \sup\{\delta_{M_{\xi}} \colon \xi \leq \gamma_p \land \delta_{M_{\xi}} < \gamma \}.$$
 As $\mathcal{M}_p$ is continuous, there exists $\eta < \gamma_p$ such that $\delta_{M_{\eta}}= \delta$. By the  assumption $p \Vdash$``$\gamma \notin \dot{C}$'', $\delta < \gamma < \delta_{M_{\eta+1}}$, hence there is $\zeta \in M_{\eta+1}$ such that $\gamma < \zeta$. Let $q$ be such that, $\mathcal{M}_q=\mathcal{M}_p$, $f_q(M_{\xi})=f_p(M_{\xi})$ for all ${\xi} \neq {\eta}$, $f_q(M_{\eta})=f_p(M_{\eta}) \cup \{\zeta\}$, and $\mathcal{W}_q= \mathcal{W}_p$. It is easily seen that $q$ is a condition, the main point is that $\mathcal{W}_q \subseteq \mathcal{M}_q=\mathcal{M}_p$, hence $\zeta \in N$ for all $N \in \mathcal{W}_q$ with $M_\eta \in N,$ in particular, $q{\restriction}_N \in N$ for  all $N \in \mathcal{W}_q$. Clearly $q$ extends $p$, and every condition extending $q$ forces that ``$\dot{C} \cap (\delta_{M_{\eta}}, \zeta] = \emptyset$'' which guaranties that $\gamma$ can not be a limit point of $C$.
\end{proof}
\begin{theorem}\label{main}
$\mathbb{P}[\alpha]$ is $\beta$-proper for every $\beta < \alpha$, but is not $\alpha$-proper.
\end{theorem}
The proof of theorem \ref{main} will consist of a series of lemmas. In fact, for every $\beta < \alpha$, and every $\beta$-tower $\mathcal{N}= \langle N_\zeta: \zeta \leq \beta \rangle$ with $\mathbb{P}[\alpha] \in N_0$, and every condition $p \in \mathbb{P}[\alpha] \cap N_0$, we will find a condition $p' \leq p$ such that:
\begin{enumerate}
\item[$(\ast)^{\beta, \mathcal{N}}_{p'}$:] for every $\zeta \leq \beta$,
\begin{enumerate}

\item  if $\zeta$ is a successor ordinal, then $p'$ is a strongly $(N_{\zeta}, \mathbb{P}[\alpha])$-generic condition, and

\item if $\zeta$ is a limit ordinal, then $p'$ is an $(N_{\zeta}, \mathbb{P}[\alpha])$-generic condition.

 \end{enumerate}
\end{enumerate}
 Also we show that for any $\alpha$-tower $\mathcal{N}$ with $\mathbb{P}[\alpha], \dot{C} \in N_0$, no condition $q \in \mathbb{P}[\alpha]$ is an $(\mathcal{N}, \mathbb{P}[\alpha])$-generic condition.

First, let us show that $\mathbb{P}[\alpha]$ is $\beta$-proper for every $\beta < \alpha$.
Fix an arbitrary $\beta < \alpha$ and let  $\mathcal{N}=\langle N_{\zeta} \colon \zeta \leq \beta \rangle$ be a $\beta$-tower of elements of $\mathcal{S}$ such that $\mathbb{P}[\alpha]  \in N_0$. Let $p \in \mathbb{P}[\alpha] \cap N_0$. Set $p'= \langle \mathcal{M}_{p'}, f_{p'}, \mathcal{W}_{p'} \rangle$, where:
 \begin{enumerate}
  \item $\mathcal{M}_{p'}= \mathcal{M}_p \cup \mathcal{N}$,
  \item $f_{p'}(M^{p'}_{\xi})=f_{p'}(M^p_{\xi})=f_p(M^p_{\xi})$ for $\xi \leq \gamma_p$,
  \item $f_{p'}(M^{p'}_{\gamma_p+1+\zeta})=f_{p'}(N_{\zeta})= \delta_{N_{\zeta}}+1$, for $\zeta \leq \beta$, and
  \item $\mathcal{W}_{p'}= \mathcal{W}_p \cup \{N_{\zeta} \in \mathcal{N} \colon \zeta \text{ is not a limit ordinal}\}$.
 \end{enumerate}
  We assert that $p'$ witnesses $(\ast)^{\beta, \mathcal{N}}_{p'}$ holds. By the construction, it is clear that $p' \leq p$, but we have to show that $p' \in \mathbb{P}[\alpha]$.
\begin{lemma}
$p'$ is a condition.
\end{lemma}
\begin{proof}
Since $p \in N_0$, so in particular $M_{\gamma_p} \in N_0$ and $f_p(M^p_{\gamma_p}) \in N_0$. Hence $\mathcal{M}_{p'}$ is an $\in$-increasing continuous sequence of elements of $\mathcal{S}$ of length $\gamma_p+\beta$. As $\alpha$ is indecomposable and  $\gamma_p$ and $\beta$ are less than $\alpha$, the length  of $\mathcal{M}_{p'}$ which is equal to $\gamma_p+\beta$ is also less than $\alpha$. Also by the construction, $f_{p'}(N_{\zeta}) \in N_{\zeta +1}$ for every $N_{\zeta} \in \mathcal{N}$.

It is enough to show that $\mathcal{W}_{p'}$ satisfies the desired requirement. Let $Q \in \mathcal{W}_{p'}$. If $Q \in \mathcal{W}_p$, then since $\mathcal{W}_p \in N_0$, we have $p'{\restriction}_Q=p{\restriction}_Q \in Q$. Now let $Q \in \mathcal{W}_{p'}\setminus \mathcal{W}_p$. Thus either $Q=N_0$ or  $Q=N_{\zeta +1}$ for some $\zeta < \beta$. If $Q=N_0$, then $p'{\restriction}_{N_0}=p \in N_0$ and we are done. Now assume that $Q=N_{\zeta+1}$, for some $\zeta < \beta$. In this case, $\mathcal{M}_{p'} \cap N_{\zeta +1}$ is the initial segment of $\mathcal{M}_{p'}$ with the last element $N_{\zeta}$, which belongs to $N_{\zeta +1}$. Also for all of the $Q' \in \mathcal{M}_{p'} \cap N_{\zeta +1}$, $f_{p'}(Q') \in N_{\zeta + 1}$, since $\mathcal{M}_{p'}$ is an $\in$-increasing continuous sequence of countable elementary substructures. As $\mathcal{W}_p \in N_0$, and $\mathcal{N}$ is an $\in$-increasing continuous sequence of countable models, so $\mathcal{W}_{p'} \cap N_{\zeta+1} \in N_{\zeta + 1}$.
\end{proof}
\begin{lemma}
$p'$ is strongly $(N_{\zeta}, \mathbb{P}[\alpha])$-generic condition for every non-limit ordinal $\zeta \leq \beta$.
\end{lemma}
\begin{proof}
We will proof this lemma in  a series of claims. Thus suppose that $\zeta \leq \beta$ is not a limit ordinal.
\begin{claim}\label{N_0}
If $q \leq p'$, then $q{\restriction}_{N_\zeta} \in \mathbb{P}[\alpha] \cap N_\zeta$.
\end{claim}
\begin{proof}[Proof of the Claim]
The proof is straightforward, the key point is that since $\mathcal{W}_{p'} \subseteq \mathcal{W}_q$, so $N_\zeta \in \mathcal{W}_q$,
which in particular implies that $q{\restriction}_{N_\zeta} \in N_\zeta$.
\end{proof}
\begin{claim}\label{sN_0}
Assume $D \subseteq \mathbb{P}[\alpha] \cap N_\zeta$ is dense open, and suppose that $r \in D$ is an extension of $q{\restriction}_{N_\zeta}$. Let $s=\langle \mathcal{M}_s, f_s, \mathcal{W}_s \rangle$ be such that $\mathcal{M}_s= \mathcal{M}_r \cup \mathcal{M}_q$, $f_s{\restriction}_{\mathcal{M}_r}= f_r$, $f_s{\restriction}_{\mathcal{M}_q \setminus \mathcal{M}_r}=f_q$ and $\mathcal{W}_s= \mathcal{W}_r \cup \mathcal{W}_q$. Then $s$ is a condition which extends both of $r$ and $q$.
\end{claim}
\begin{proof}[Proof of the Claim]  By construction, it is easy to see $s$ is a common extension of $r$ and $q$, so we just have to show that $s \in \mathbb{P}[\alpha]$.
 Since the order types of both $\mathcal{M}_r$ and $\mathcal{M}_q$ are less than the indecomposable ordinal $\alpha$, so is the order type of $\mathcal{M}_s$ (which is at most $\gamma_r + \gamma_q$). Furthermore, $\mathcal{M}_s$ is an $\in$-increasing continuous chain, of the form
 \[
 \mathcal{M}_s=\mathcal{M}_r ^{\frown} \langle N_\zeta \rangle ^{\frown} \langle N \in \mathcal{M}_q: N_\zeta \in N \rangle.
 \]
  Note that $r \in N_\zeta$ gives us $\mathcal{M}_r \in N_\zeta$, in particular $M^r_{\gamma_r} \in N_\zeta$ and $f_s(M^s_{\gamma_r})=f_r(M^r_{\gamma_r}) \in N_\zeta$. Also, $f^s(N_\zeta)=f^q(N_\zeta) \in N$, for the least $N \in \mathcal{M}_q$ with $N_\zeta \in N$.  It immediately follows that for every $\xi < \gamma_s,$  $f_s(M^s_{\xi}) \in M^s_{\xi+1}$.

  Now let $Q \in \mathcal{W}_s$. If $Q \in \mathcal{W}_r$, then $s{\restriction}_Q=r{\restriction}_Q \in Q$. If $Q \in \mathcal{W}_q \setminus \mathcal{W}_r$, then $s{\restriction}_Q= s{\restriction}_{N_\zeta} \cup s{\restriction}_{(N_{\zeta}, Q)}=r \cup q{\restriction}_{(N_{\zeta}, Q)}=r \cup q{\restriction}_Q$ which  clearly is in $Q$.
\end{proof}
It follows from the above claim that $p'$ is  a strongly $(N_\zeta, \mathbb{P}[\alpha])$-generic condition.
\end{proof}
The next lemma is well-known, which takes care of the limit steps.
\begin{lemma}
Suppose  $\zeta \leq \beta$ is a limit ordinal and for every ordinal $\eta$ less than $\zeta$, $p'$ is $(N_{\eta}, \mathbb{P}[\alpha])$-generic condition. Then $p'$ is an $(N_{\zeta}, \mathbb{P}[\alpha])$-generic condition.
\end{lemma}
\begin{proof}
Let $\dot{\tau} \in N_{\zeta}$ be a name of an ordinal. Since $N_{\zeta}= \bigcup_{\eta < \zeta}N_{\eta}$, so $\dot{\tau} \in N_{\eta}$ for some $\eta < \zeta$. By assumption, $p' \Vdash `` N_{\eta} \cap \on = N_{\eta}[\dot{G}] \cap \on"$, i.e. $p' \Vdash `` \dot{\tau} \in N_{\eta}"$ which implies $p' \Vdash `` \dot{\tau} \in N_{\zeta}$. Since $\dot{\tau}$ is arbitrary, so $p'$ is $(N_{\zeta}, \mathbb{P}[\alpha])$-generic condition.
\end{proof}
We now show that $\mathbb{P}[\alpha]$ is not $\alpha$-proper.
\begin{lemma}
$\mathbb{P}[\alpha]$ is not $\alpha$-proper.
\end{lemma}
\begin{proof}
Let $\mathcal{N}=\langle N_{\zeta} \colon \zeta \leq \alpha \rangle$ be an arbitrary $\alpha$-tower,with $\mathbb{P}[\alpha], \dot{C} \in N_0$,
 where $\dot{C}$ is the canonical name for the club added by the forcing. Let $p \in \mathbb{P}[\alpha] \cap N_0$, and assume towards contradiction that $q \leq p$ is an $(\mathcal{N}, \mathbb{P}[\alpha])$-generic condition. Without loss of generality, assume $\delta_{N_{\alpha}} < \delta_{M^q_{\gamma_q}}$. For every $\zeta \leq \alpha$, as   $\dot{C} \in N_\zeta$ and $q$ is an $(N_\zeta, \mathbb{P}[\alpha])$-generic condition, $q$ forces $\dot{C} \cap \delta_{N_{\zeta}}$ is unbounded in $\delta_{N_{\zeta}}$,  hence $q \Vdash$``$ \delta_{N_{\zeta}} \in \dot{C}$''.

As $\alpha$ is indecomposable, and the set $\mathcal{M}_q$ has order type less than $\alpha,$ we can find some ordinal $\zeta < \alpha$
such that
\[
\{\delta_M: M \in \mathcal{M}_q\} \cap [\delta_{N_\zeta}, \delta_{N_{\zeta+2}})  = \emptyset.
\]
Let $\eta < \gamma_q$ be maximal such that $\delta_{M^q_\eta} < \delta_{N_\zeta}.$ It then follows that $\delta_{M^q_{\eta+1}} \geq  \delta_{N_{\zeta+2}}.$ As  $q \Vdash$``$ \delta_{N_{\zeta}} \in \dot{C}$'', we can find some model $\tilde{M}$
such that:
\begin{itemize}
\item $\tilde{M} \in M^q_{\eta+1}$,
\item $M^q_\eta \in \tilde{M},$
\item $\delta_{\tilde{M}}= \delta_{N_\zeta}.$
\end{itemize}
let $r$ be such that $\mathcal{M}_{r}= \mathcal{M}_q \cup \{\tilde{M}\}$, $f_{r}(M)=f_q(M)$ for all $M \in \mathcal{M}_q$,  $f_{r}(\tilde{M})= \{\delta_{N_{\zeta+1}}+1\}$ and $\mathcal{W}_{r}= \mathcal{W}_q$. Then $r$ is a condition. The main point is that for all $N \in \mathcal{W}_{r}$,
if $\tilde{M} \in N$, then $\delta_{N_{\zeta+1}}+1 \in N$, as $\delta_{N_{\zeta+1}}+1 < \delta_{N_{\zeta+2}} \leq \delta_N$. Clearly $r$ is an extension of $q$ and it  forces that $\delta_{N_{\zeta+1}} \notin \dot{C}$, which is a contradiction.
\end{proof}
Theorem \ref{main} follows.

\end{document}